\documentclass[11pt,a4paper]{article}
\usepackage{authblk}
\usepackage{hyperref}
\usepackage{amssymb}
\usepackage{amsthm}
\usepackage{amsmath}
\usepackage{multirow}
\usepackage{float}
\usepackage{graphicx}
\newtheorem{theorem}{Theorem}[section]

\newtheorem{lemma}[theorem]{Lemma}
\newtheorem{proposition}[theorem]{Proposition}

\newtheorem{definition}[theorem]{Definition}
\newtheorem{example}[theorem]{Example}

\newcommand{\R}{\mbox{$\Bbb R$}} 
\title{ \Large\bfseries A novel procedure for constructing invariant subspaces of a set of matrices}
\author[1]{Ahmad Y. Al-Dweik}
\author[2]{Ryad Ghanam}
\author[3]{Gerard Thompson}
\author[4]{Hassan Azad}
\affil[1]{Department of Mathematics, Statistics and Physics, Qatar University, Doha, 2713, State of Qatar; aydweik@qu.edu.qa}
\affil[2]{Department of Liberal Arts $\&$ Sciences, Virginia Commonwealth University in Qatar, Doha 8095, Qatar; raghanam@vcu.edu}
\affil[3]{Department of Mathematics, University of Toledo, Toledo, OH 43606, USA; gerard.thompson@utoledo.edu}
\affil[4]{Abdus Salam School of Mathematical Sciences, GC University, Lahore 54600, Pakistan; hassan.azad@sms.edu.pk}
\begin{document}
\maketitle
\begin{abstract}
A problem that is frequently encountered in a variety of mathematical contexts, is to find the common invariant subspaces of a single, or set of
matrices. A new method is proposed that gives a definitive answer to this problem. The key idea consists of finding common eigenvectors for
exterior powers of the matrices concerned. A convenient formulation of the Pl\"ucker relations is then used to ensure that these eigenvectors actually correspond to
subspaces or provide the initial constraints for eigenvectors involving parameters. A procedure for computing the divisors of  totally decomposable vector is also provided. Several examples are given for which the calculations are too tedious to do by hand and are performed by coding the conditions
found into Maple.
\end{abstract}
\bigskip


\bigskip

\section{Introduction}

A problem that occurs frequently in a variety of mathematical contexts, is to find the common invariant subspaces of a single, or set of
matrices. Of course the problem is even more challenging in infinite dimensions and the invariant subspace problem
for bounded operators on a separable Hilbert space remains open; see for example, \cite{E,Y}. In this article we shall be
concerned with finite dimensions only. Although it is usually easy to find lots of common invariant subspaces for one or more matrices, there
remains the problem of ensuring that one has obtained \emph{all} possible such subspaces. In this article we propose a novel method that will
overcome such difficulties.

In \cite{T}, Tsatsomeros provided a necessary and sufficient conditions for the existence of a common non-trivial invariant subspace of a set of matrices. The author also provided a six step method for finding such subspaces for two matrices.
He states that step 4 of the method can be performed using Algorithm 12.4.3 in \cite{GVL} and that steps 5 and 6 can prove to be theoretically and practically challenging. 

In this work, we extend his method to a finit set of matrices. To avoid the difficulty in step 4, we provide a different algorithm for computing bases for the intersections of eigenspaces of a set of matrices. 

To overcome the challenges in steps 5  and 6, of checking if the intersecting eigenspaces of two matrices contain a non-zero decomposable vector and find its divisors, we provide a convenient formulation of the Pl\"ucker relations that can be used to check for the decomposability of a multivector $\Lambda \in \bigwedge^d V$. If the  multivector $\Lambda$ involves parameters, then the quadratic Pl\"ucker relations for decomposability provide initial constraints  on the these parameters so as to keep $\Lambda$ totally decomposable in  $\bigwedge^d V$.
The Plücker relations are given in a simple form which is easily programmable using the symbolic
manipulation program Maple. Moreover, we provide a procedure for computing the divisors of  totally decomposable vector
$\Lambda \in \bigwedge^d V$.

An outline of the paper is as follows. In Section 2 we give some details about invariant subspaces, in particular, explaining the connection between
an invariant subspace and a totally decomposable multivector. In Section 3 we provide a low-dimensional example, where the details can be carried out by hand.
In Section 4 we present the main result of the paper, which gives a convenient formulation of the Pl\"ucker relations that can be used to
check for the decomposability of a multivector or provide the initial constraints for a multivector involving parameters. A procedure for computing the divisors of a totally decomposable vector $\Lambda
\in \bigwedge^d V$ is also provided. In Section 5 we outline two algorithms to find common invariant subspaces, first for one-dimensional
subspaces then for higher dimensions. Finally, in Section 6 we provide three examples for which the calculations are too complicated to do by hand
and are performed using Maple although we do not go into details of the code here.
\section{Invariant Subspaces}
Let $\{v_1, v_2, ..., v_n\}$ be a basis of an $n$-dimensional vector space $V$.
We shall denote the associated dual basis for the dual space $V^*$
by $\omega_1,\omega_2,...,\omega_n$, so that $\omega_i(v^j)=\delta^j_i$ or, equivalently, we shall
write $\langle v^j, \omega_i \rangle=\delta^j_i$.
For the moment we shall assume either that the underlying field of $V$ is either
$\R$ or else that all eigenvalues of the matrices encountered are real. We shall address the issues of complex eigenvalues
in a subsequent article. Now suppose that $T_A:V\rightarrow V$ ia an endomorphism of $V$.
Since we have chosen a basis for $V$, we shall identify $T_A$ with its $n\times n$ matrix, denoted by $A$. Thus we have
\begin{equation}
T_A(v_i)=\sum\limits_{j=1}^n a_i^jv_j.
\end{equation}

\begin{definition}\rm \label{d2}
A subspace $W\subset V$
is said to be \emph{invariant} with respect to the transformation $T_A$ if
$T_AW\subset W$.
\end{definition}
An extensive study of invariant subspaces of transformations may be found in \cite{GLR}. See also \cite{BF}.
Let $T_{A_{\alpha}}$ be a family of linear transformations indexed by $\alpha$ and suppose that
$W_1, W_2$ are  $T_{A_{\alpha}}$-invariant. Then
\begin{equation}
T_{A_{\alpha}}(W_1\cap W_2)\subset T_{A_{\alpha}}W_1\cap T_{A_{\alpha}}W_2\subset W_1\cap W_2
\end{equation}
and
\begin{equation}
T_{A_{\alpha}}(W_1+W_2)=T_{A_{\alpha}}W_1+T_{A_{\alpha}}W_2\subset W_1 + W_2.
\end{equation}
The meet of two such subspaces $W_1, W_2$ is $W_1\cap W_2$ and their join is the space $W_1+W_2$, that is, the space
spanned by $W_1$ and $W_2$. Hence the set of invariant subspaces for a family of transformations form a lattice.

It is clear that if $T_A$ is a multiple of the identity transformation on $V$, then \emph{every}
subspace of $V$ is invariant. More generally, if a family of transformations possesses a subspace $W$ of
$V$ on which each of them is a multiple of the identity, then again any subspace of $W$ is invariant.
Accordingly, we shall sometimes assume in the sequel that no such invariant subspace exists.

Our goal is to find, if possible, the lattice of invariant subspaces of a single transformation
and eventually for a family of transformations. We shall specify these subspaces by giving
decomposable multivectors $\Lambda$ in terms of the reference basis $\{v_1, v_2, ..., v_n\}$ used for $V$.
Such a $\Lambda$ of degree $d$ in $\bigwedge^d V$, the $d^{th}$ exterior power of
$V$. A basis for $\bigwedge^d V$ consists of $\{v_{i_1}\wedge v_{i_2}\wedge...\wedge
v_{i_d}\mid 1\leq i_1 < i_2 <...<i_d \leq n\}$ where $1\leq d \leq n$.

However, with reference to a single transformation $T_A$, we shall frequently assume that the
basis $\{v_1, v_2, ..., v_n\}$ is \emph{adapted} to a $d$-dimensional invariant subspace $W$, by which we mean
that $\{v_1,v_2, ..., v_d\}$ is a basis for $W$. Then $T_A$ induces
an eigenvector in $\bigwedge^d V$, the $d$th exterior power of $V$. Indeed, in the adapted basis
\begin{equation}\label{wed}
\bigwedge^d T_A(v_1\wedge v_2\wedge ...\wedge v_d)=\det(\tilde{A})(v_1\wedge v_2\wedge ...\wedge v_d).
\end{equation}
In eq.(\ref{wed}), $\tilde{A}$ denotes the submatrix of $A$ that results when $T_A$ is restricted to $W$.
At this point it is convenient to introduce the following definition.
\begin{definition}\rm \label{d3}
The vector $\Lambda \in \bigwedge^d V$ is said to be \emph{totally
decomposable} if there are $d$ linearly independent vectors $v_1, v_2,
...,v_d \in V$ such that $\Lambda =v_1 \wedge v_2 \wedge ...\wedge v_d$.
\end{definition}
\noindent Again, it may be assumed in Definition (\ref{d3}) that we are working with a basis of $V$
that is adapted to $W$. If we drop the independence condition in Definition (\ref{d3}), then we would say simply that
$\Lambda$ is decomposable. Thus, an invariant subspace engenders a totally
decomposable multivector. Conversely, suppose that $\Lambda =v_1\wedge v_2\wedge ...\wedge v_d$,
is a a totally decomposable eigenvector of $\bigwedge^d A$; then
\begin{equation}
\bigwedge^d A(v_1\wedge v_2\wedge...\wedge v_d)=\lambda  (v_1\wedge v_2\wedge ...\wedge v_d)
\end{equation}
for some $\lambda$. Since $\bigwedge^d A(v_1\wedge v_2\wedge ...\wedge v_d)= Av_1\wedge  Av_2\wedge ...\wedge Av_d$, then we have $Av_1\wedge  Av_2\wedge ...\wedge Av_d=\lambda  (v_1\wedge v_2\wedge ...\wedge v_d)$.
If we assume that $\lambda\ne 0$, then  $Av_m\wedge (v_1\wedge v_2\wedge ...\wedge v_d)=0$ for each $1 \leq m \leq d$
and hence the subspace spanned by $v_1, v_2, ...,v_d$ is $A$-invariant or $T_A$-invariant.
Thus, we have the key observation that underlies this paper.
\begin{proposition}
Assuming that the matrix $A$ is non-singular, there is a one-one correspondence between invariant subspaces of $A$ and projective equivalence classes of totally decomposable eigenvectors of $\bigwedge^d A$
where $1\leq d \leq n$.
\end{proposition}

It is necessary to assume that the matrix $A$ is non-singular. For example, if $A$ is nilpotent then some exterior power $\bigwedge^k A=0$
and we may not obtain information about invariant subspaces of $A$. On the other hand, we can always add a suitable multiple of the identity
to $A$ so as to obtain $\overline{A}$, which is non-singular: $A$ and $\overline{A}$ have the same invariant subspaces. More generally:
\begin{lemma}
Given a set of $n\times n$ matrices $\{A_1,A_2,...,A_k\}$, it is possible to add a suitable multiple of the identity
to each $A_i$ so as to obtain $\overline{A_i}=A_i+\mu_iI$, each of which is non-singular and then the sets
$\{A_1,A_2,...,A_k\}$ and $\{\overline{A}_1,\overline{A}_2,...,\overline{A}_k\}$ have the same invariant subspaces.
\end{lemma}
\begin{proof}
Since we are only considering a finite number of matrices, there are only a finite number of eigenvalues altogether for the
matrices $A_1,A_2,...,A_k$. We simply choose each $\mu_i$ so that each $A_i+\mu_iI$ is non-singular. Adding
multiples of the identity does not change the common invariant subspaces of $A_1,A_2,...,A_k$.
\end{proof}

\section{Illustrative Example}

In this Section we study an example where the calculations are simple enough to do by hand.
We consider the problem of finding the common invariant subspaces of the following three matrices:

\begin{equation}\label{41}
A_1= \left( \begin {array}{cccc}
0&0&0&1\\ \noalign{\medskip}
0&0&0&0\\ \noalign{\medskip}
0&0&0&0\\ \noalign{\medskip}
 \noalign{\medskip}
0&0&0&0\end {array} \right),
A_2= \left( \begin {array}{cccc}
0&0&0&0\\ \noalign{\medskip}
0&0&0&1\\ \noalign{\medskip}
0&0&0&0\\ \noalign{\medskip}
 \noalign{\medskip}
0&0&0&0\end {array} \right),
A_3= \left( \begin {array}{cccc}
0&1&0&0\\ \noalign{\medskip}
0&0&1&0\\ \noalign{\medskip}
0&0&0&0\\ \noalign{\medskip}
 \noalign{\medskip}
0&0&0&0\end {array} \right).
\end{equation}
In the first place we note that $\bigwedge^2 A_1=\bigwedge^2 A_2=0$. However, the common invariant subspaces of $A_1,A_2,A_3$
are not just determined by $A_3$ only. Accordingly, we add the identity matrix to each of $A_1, A_2, A_3$ so as to obtain $\overline{A}_1, \overline{A}_2, \overline{A}_3
$ giving
\begin{equation}\label{410}
\overline{A}_1= \left( \begin {array}{cccc}
1&0&0&1\\ \noalign{\medskip}
0&1&0&0\\ \noalign{\medskip}
0&0&1&0\\ \noalign{\medskip}
 \noalign{\medskip}
0&0&0&1\end {array} \right),
\overline{A}_2= \left( \begin {array}{cccc}
1&0&0&0\\ \noalign{\medskip}
0&1&0&1\\ \noalign{\medskip}
0&0&1&0\\ \noalign{\medskip}
 \noalign{\medskip}
0&0&0&1\end {array} \right),
\overline{A}_3= \left( \begin {array}{cccc}
1&1&0&0\\ \noalign{\medskip}
0&1&1&0\\ \noalign{\medskip}
0&0&1&0\\ \noalign{\medskip}
 \noalign{\medskip}
0&0&0&1\end {array} \right).
\end{equation}

Then we find, taking the basis in the order $e_1\wedge e_2, e_1\wedge e_3, e_1\wedge e_4, e_2\wedge e_3, e_2\wedge e_4, e_3\wedge e_4$, that

\begin{equation}\label{411}
\bigwedge^2\overline{A}_1=
\left( \begin{smallmatrix}
1&0&0&0&-1&0\\ \noalign{\medskip}
0&1&0&0&0&-1\\ \noalign{\medskip}
0&0&1&0&0&0\\ \noalign{\medskip}
0&0&0&1&0&0\\ \noalign{\medskip}
0&0&0&0&1&0\\ \noalign{\medskip}
0&0&0&0&0&1
\end {smallmatrix} \right),
\bigwedge^2\overline{A}_2=
\left( \begin {smallmatrix}
1&0&1&0&0&0\\ \noalign{\medskip}
0&1&0&0&0&0\\ \noalign{\medskip}
0&0&1&0&0&0\\ \noalign{\medskip}
0&0&0&1&0&-1\\ \noalign{\medskip}
0&0&0&0&1&0\\ \noalign{\medskip}
0&0&0&0&0&1
\end {smallmatrix} \right),
\bigwedge^2\overline{A}_3=
\left( \begin {smallmatrix}
1&1&0&1&0&0\\ \noalign{\medskip}
0&1&0&1&0&0\\ \noalign{\medskip}
0&0&1&0&1&0\\ \noalign{\medskip}
0&0&0&1&0&0\\ \noalign{\medskip}
0&0&0&0&1&1\\ \noalign{\medskip}
0&0&0&0&0&1
\end {smallmatrix}\right).
\end{equation}

Using the basis
$e_1\wedge e_2 \wedge e_3, e_1\wedge e_2 \wedge e_4, e_1\wedge e_3 \wedge e_4, e_2\wedge e_3 \wedge e_4$

\begin{equation}\label{412}
\bigwedge^3 \overline{A}_1= \left( \begin {array}{cccc}
1&0&0&1\\ \noalign{\medskip}
0&1&0&0\\ \noalign{\medskip}
0&0&1&0\\ \noalign{\medskip}
\noalign{\medskip}
0&0&0&1\end {array} \right),
\bigwedge^3 \overline{A}_2= \left( \begin {array}{cccc}
1&0&-1&0\\ \noalign{\medskip}
0&1&0&0\\ \noalign{\medskip}
0&0&1&0\\ \noalign{\medskip}
 \noalign{\medskip}
0&0&0&1\end {array} \right),
\bigwedge^3 \overline{A}_3=\left( \begin {array}{cccc}
1&0&0&0\\ \noalign{\medskip}
0&1&1&1\\ \noalign{\medskip}
0&0&1&1\\ \noalign{\medskip}
 \noalign{\medskip}
0&0&0&1\end {array} \right).
\end{equation}
Now we see in each of the matrices in eq.(\ref{410}) and eq.(\ref{411}) that the first column gives the only common
eigenvector. Accordingly, the only common invariant subspaces of dimensions one and two are $\langle e_1 \rangle$ and
$\langle e_1,e_2 \rangle$, respectively. On the other hand, from eq.(\ref{412}), we see that there are \emph{two}
linearly independent eigenvectors corresponding to the first and second columns of each matrix. It follows that
any linear combination of these eigenvectors will also produce a common invariant subspace, which will be of the
form $\langle e_1,e_2, ae_3+be_4 \rangle, (a^2+b^2\neq0)$. In particular, we see that there are an infinite number of
invariant subspaces of dimension three.

The reader will observe that the original example considered in this Section arises from the adjoint representation
of the unique four-dimensional nilpotent Lie algebra. However, one needs to exercise caution when adding multiples of the identity
to the generators, because they may no longer span a subalgebra.

\section{Pl\"ucker embedding and decomposability}

As a result of the considerations of the Section 2, it will be important to decide whether a given multivector $\Lambda$
is totally decomposable. In the case of a bivector, it is decomposable if and only if $v\wedge v=0$. On the other
hand, if $v$ is of odd degree, then $v\wedge v=0$ holds identically. Furthermore, multivectors of degree one, $n-1$ and
$n$ are always decomposable. However, to handle  multivectors of arbitrary degree, we shall have to introduce some
more machinery.

We continue with our vector space $V$ of dimension $n$.
We let $G(d,V)$ denote the Grassmann manifold of all $d$-planes in $V$.
Then $G(d,V)$ is a smooth manifold of dimension $d(n-d)$.
Pl\"ucker embedding consists of mapping $G(d,V)$ to the projective space
$P(\bigwedge^d V)$; to do so, let $W\in G(d,V)$ and let
$v_1,v_2,...,v_d$ be a basis for $W$. Then map $W$ to
$[v_{1} \wedge v_{2} \wedge... \wedge v_{d}]$, where the parentheses
denote equivalence in $P(\bigwedge^d V)$. The mapping is well defined;
if we take a different basis and wedge its vectors together, the wedge product
differs from the old one by a non-zero factor, that is the determinant of the
change of basis, and so will yield the same element of $P(\bigwedge^d V)$.
Moreover, this mapping is injective, since every totally decomposable
element of $P(\bigwedge^d V)$ arises from a unique subspace.

There are many equivalent ways to characterize total decomposability in terms of the
so called Pl\"ucker relations; see for example, \cite{H,HP,M}. We shall obtain next a convenient formulation
of these conditions. To begin, note that whenever a volume element of $V$ is given, that is, a non-zero
element of $\bigwedge^n V$, there an isomorphism between $\bigwedge^d V$ and
$\bigwedge^{n-d} V^{*}$. As such, we shall use the volume element $v_N=v_{1} \wedge v_{2} \wedge ... \wedge v_{n}$ coming
from our reference basis and we shall denote by $\bar{\Lambda}$ the element $\in \bigwedge^{n-d} V^{*}$ associated to
$\Lambda \in \bigwedge^d V$. Now, we can assert that $\Lambda$ is totally decomposable if and only if its components
satisfy the following Pl\"ucker relations
\begin{equation}\label{PR2}
\begin{array}{cc}
E_{KL}=\sum\limits_{s=1}^n {\omega}_{K}(v_s \wedge \Lambda)   ({\omega}_s \wedge \bar{\Lambda})v_{L}=0,\\
\end{array}
\end{equation}
where $K$ and $L$ are strictly increasing subsequences of
$N=(1,...,n)$ such that $|K|=d+1$ and $|L|=n-d+1$.

In fact it is useful to evaluate $\bar{\Lambda}$ directly in terms of $\Lambda$ as we shall now do.
If $\Lambda=\sum\limits_{|I|=d} x_{I} v_{I} \in \bigwedge^d
V$, then $\bar{\Lambda}=\sum\limits_{|J|=n-d}
y_{J}{\omega}_{J} \in \bigwedge^{n-d} V^{*}$ and
$\Lambda \wedge v_{J}=y_{J} v_N$. Therefore we have $\sum\limits_{|I|=d}
x_{I} v_{I} \wedge v_{J}=y_{J} v_N.$ Now, since $v_I \wedge v_{J}
\ne 0 \Longleftrightarrow I=J'$, where the sequence $I$ is increasing and
complementary to $J.$ Moreover, $v_{J'} \wedge
v_{J} =\text{sgn}\left(J',J\right) v_N$ where $\text{sgn}(\sigma)$
is the signature of the permutation $\sigma$. It follows that
$y_{J}=\text{sgn}\left(J',J\right) x_{J'}$. So
\begin{equation}\label{ww}
\bar{\Lambda}=\sum\limits_{|J|=n-d}\text{sgn}\left(J',J\right) x_{J'}
{\omega}_{J}.
\end{equation}
\subsection{ A convenient formulation of the Pl\"ucker relations}
Now we are in a position to obtain a version of the Pl\"ucker relations that is suitable for use in our algorithm.
\begin{theorem}\rm
Let $V$ be an $n$-dimensional vector space.  The vector $\Lambda
\in \bigwedge^d V$ is totally decomposable if and only
if its components satisfy the following Pl\"ucker relations
\begin{equation}\label{PR1}
\begin{array}{cc}
E_{KL}=\sum\limits_{s\in K \cap L} \text{sgn}\left(s,K \setminus \{s\}\right)\cdot\text{sgn}\left(s,L \setminus \{s\}\right)\cdot\text{sgn}\left((L \setminus \{s\})',L \setminus \{s\}\right)\\
\cdot  x_{K \setminus \{s\}}\cdot x_{(L \setminus \{s\})'}=0,\\
\end{array}
\end{equation} where $K, L$ are strictly
increasing subsequences of $N=(1, 2,...,n)$ such that $\mid
K\mid=d+1, \mid L\mid=n-d+1$, the sequence $(L \setminus \{s\})'$
denotes the increasing sequence complementary to $L \setminus \{s\}$ and
$\text{sgn}(\sigma)$ is the signature of the permutation $\sigma$.
\end{theorem}
\proof Let $I$ and $J$ be
strictly increasing subsequences of $N=(1, 2,...,n)$ such that
$|I|=d$ and $|J|=n-d$. For the subsequences $I:i_1<i_2<...<i_d$ and
$J:j_1<j_2<...<j_{n-d}$, we write $v_{i_1} \wedge v_{i_2} \wedge ... \wedge
v_{i_d}=v_I$ and $\omega_{j_1} \wedge \omega_{j_2} \wedge... \wedge
\omega_{j_{n-d}}=\omega_J$.

Then adapting the formulas for $\Lambda$ and $\bar{\Lambda}$ from equation (\ref{ww}) in equation (\ref{PR2}), gives
\begin{equation}\label{PR12}
\begin{array}{cc}
E_{KL}=\sum\limits_{s=1}^n \sum\limits_{|I|=d} x_{I}{\omega}_{K}(v_s \wedge v_{I}) \sum\limits_{|J|=n-d} \text{sgn}\left( J',J\right) x_{J'}({\omega}_s \wedge {\omega}_{J})v_{L}=0.\\
\end{array}
\end{equation}

Now ${\omega}_{K}(v_s \wedge v_{I}) \ne 0
\Longleftrightarrow I=K\setminus
 \{s\}$ and $s\in K$. Similarly,  $({\omega}_s \wedge {\omega}_{J})v_{L} \ne 0 \Longleftrightarrow J=L\setminus
 \{s\}$ and $s\in L$. Thus equation (\ref{PR2}) can be rewritten as
\begin{eqnarray}
&& \sum\limits_{s\in K \cap L}  \text{sgn}\left( (L\setminus\{s\})',L\setminus\{s\}\right) x_{(L\setminus\{s\})'} x_{K\setminus \{s\}}{\omega}_{K}(v_s \wedge v_{K\setminus\{s\}})  ({\omega}_s \wedge {\omega}_{L\setminus \{s\}})v_{L}\nonumber\\
&& = E_{KL}=0.\label{PR3}
\end{eqnarray}
Finally, substituting $v_s \wedge v_{K\setminus
 \{s\}}=\text{sgn}\left(s,K \setminus \{s\}\right) v_K$ and ${\omega}_s \wedge {\omega}_{L\setminus
 \{s\}}=\text{sgn}\left(s,L \setminus \{s\}\right) {\omega}_L$ in equation (\ref{PR3}) completes the proof.
\endproof
\subsection{Procedure for computing the divisors of  totally decomposable vector $\Lambda
\in \bigwedge^d V$}
Given a totally decomposable $\Lambda \in \bigwedge^2 V$,
there exist two linearly independent vectors $u_1,
u_2 \in V$ such that $\Lambda  =u_1 \wedge u_2$.  Moreover, the
divisors $u_1, u_2$ satisfy the conditions $u_1 \wedge \Lambda  =0$ and $u_2 \wedge \Lambda  =0$. Therefore, for a given  totally decomposable vector $\Lambda =\sum\limits_{1\leq i<j \leq n} x_{ij} v_i\wedge
v_j$, if we assume that $u=\sum\limits_{k=1}^n a_k v
_k$ is a divisor of $\Lambda$, then one
can find $u$ by solving the linear system generated by comparing
the coefficients of the basis of $\bigwedge^3 V$ in the
following equation $$u \wedge \Lambda  = \sum\limits_{1\leq i<j \leq n}
\sum\limits_{k=1}^n x_{ij} a_k ~v_k \wedge ~v_i\wedge v_j=0.$$
Solving this linear system will give the two divisors $u_1,u_2 \in V$ such that $\Lambda  =u_1 \wedge u_2$.

Now, if the  vector $\Lambda \in \bigwedge^2 V$ involves some parameters, applying the Pl\"ucker relations (\ref{PR1})
provides initial constraints on the these parameters so as to keep $\Lambda$ totally decomposable in  $\bigwedge^2 V$. Moreover, the system of divisors of $\Lambda$ will inherit these parameters. Therefore, we need to solve the system of divisors with the initial constraints for all the possible cases of the parameters. This can be achieved using the command ``PreComprehensiveTriangularize($sys, d, R$)" in Maple 13, that is given in \cite{CGLMP}. This command returns a pre-comprehensive triangular decomposition of sys, with respect to the last $d$ variables of $R$.

Similarly, the above idea can be extended to find the divisors of a totally decomposable vector $\Lambda
\in \bigwedge^d V$ for $2<d<n$.\\
\section{Algorithms}
\subsection{Algorithm $A$: Determining the common one-dimensional invariant subspaces for a set of $n \times n$ matrices $\{A_i \}_{ i=1}^N$.}
\begin{enumerate}
\item Input: $\{A_i \}_{ i=1}^N$.
\item Find the set of eigenvalues $\sigma(A_i)$ for each matrix $A_i$.
\item Find $\Omega=\{(\lambda_1, \lambda_2,..., \lambda_N): \lambda_i \in \sigma(A_i)  \}.$
\item Construct the following matrix for each $(\lambda_1, \lambda_2,..., \lambda_N) \in \Omega$:
\begin{equation*}
B(\lambda_1, \lambda_2,..., \lambda_N)=
 \left( \begin {array}{c}
A_1-\lambda_1 I\\ \noalign{\medskip}
A_2-\lambda_2 I\\ \noalign{\medskip}
.\\
.\\
.\\
A_N-\lambda_N I\\ \noalign{\medskip}
\end {array} \right) .
\end{equation*}
\item Compute the  null space $\Lambda$  for each matrix $B(\lambda_1, \lambda_2,..., \lambda_N)$.
\item Output: The set of all possible $((\lambda_1, \lambda_2,..., \lambda_N), \Lambda)$.
\end{enumerate}
\subsection{Algorithm $B$:  Determining the common invariant subspaces  of dimension $1<d<n$ for a set of $n \times n$ matrices $\{A_i \}_{ i=1}^N$.}
\begin{enumerate}
\item Input: $\{A_i \}_{ i=1}^N$ and $d$.
\item  Find $s$ such that $\{A_i + sI\}_{ i=1}^N$ are invertible matrices. Let $\bar A_i=A_i + sI$ for $i=1,2,...,N$.
\item  Compute the $m \times m$ matrices $\{\bigwedge^d \bar A_i \}_{ i=1}^N$ where $m=\dim \bigwedge^d V=\binom n d$.
\item Determine the set of all common one-dimensional invariant subspaces and the corresponding sequence of eigenvalues $(\lambda, \Lambda)$ to the set of $m \times m$ matrices $\{\bigwedge^d \bar A_i \}_{ i=1}^N$ using Algorithm A.
\item Construct the Pl\"ucker relations  in $\bigwedge^d V$ using the  convenient formulation given in section (4.1).
\item For each sequence of eigenvalues $\lambda$,  check if the coefficients of the corresponding one-dimensional invariant subspace $\Lambda$ satisfy the Pl\"ucker relations. If the  vector $\Lambda \in \bigwedge^d V$ involves some parameters, apply the Pl\"ucker relations to find the initial constraints on the these parameters.
\item Find the divisors $\{v_1, v_2, ..., v_d\}$ of each totally decomposable one-dimensional invariant subspace $\Lambda=v_1 \wedge v_2 \wedge ... \wedge v_d$ using the procedure given in section (4.2).
\item Output: The set of all possible $(\lambda, \langle v_1, v_2, ..., v_d \rangle)$.
\end{enumerate}
\section{Examples}
\begin{example}
\begin{equation}
A_1= \left( \begin {array}{ccccccc}
3&0&0&0&0&0&0\\ \noalign{\medskip}
0&2&0&0&0&0&0\\ \noalign{\medskip}
0&0&2&0&0&0&0\\ \noalign{\medskip}
0&0&0&1&0&0&0\\ \noalign{\medskip}
0&0&0&0&1&0&0\\ \noalign{\medskip}
0&0&0&0&0&1&0\\ \noalign{\medskip}
0&0&0&0&0&0&3\end {array} \right),
A_2=\left( \begin {array}{ccccccc}
 0&0&0&0&0&0&0\\ \noalign{\medskip}
0&0&0&0&0&0&0\\ \noalign{\medskip}
0&1&0&0&0&0&0\\ \noalign{\medskip}
0&0&0&0&0&0&0\\ \noalign{\medskip}
0&0&0&0&0&0&0\\ \noalign{\medskip}
0&0&0&1&0&0&0\\ \noalign{\medskip}
1&0&0&0&0&0&0\end {array} \right).
\end{equation}
Using algorithms $A$ and $B$ with $s=1$, the complete list of the common invariant subspaces for the set of matrices $\{A_1,A_2\}$ is  as follows:
 \begin{itemize}
\item Zero-dimensional subspaces:\\
$\{\bf 0\}$
\item  One-dimensional subspaces:\\
\begin{tabular}{ |c|c| }
 \hline
Sequence of eigenvalues&  One-dimensional  \\
$\lambda=(\lambda_1,\lambda_2)$ for $\{ \bar A_i \}_{ i=1}^2$ &invariant subspace\\
 \hline
 (2,1)  & $\langle e_5\rangle,\langle e_6+\alpha e_5\rangle$ \\
 \hline
 (3,1) & $\langle e_3\rangle$\\
 \hline
 (4,1) & $\langle e_7\rangle$ \\
 \hline
\end{tabular}
\item  Two-dimensional subspaces:\\
\begin{tabular}{ |c|c| }
 \hline
Sequence of eigenvalues &    Two-dimensional \\
$\lambda=(\lambda_1,\lambda_2)$ for $\{\bigwedge^2 \bar A_i \}_{ i=1}^2$  & invariant subspace\\
 \hline
(4,1) &  $\langle e_4 ,e_6\rangle,\langle e_5+\alpha e_4 ,e_6\rangle$\\
 \hline
(6,1) &   $\langle e_3, e_5\rangle,\langle e_3,e_6+\alpha  e_5\rangle$\\
 \hline
(8,1) &  $\langle e_5,e_7\rangle,\langle e_6+\alpha e_5,e_7\rangle$\\
 \hline
(9,1) &   $\langle e_2,e_3\rangle$\\
 \hline
(12,1) & $\langle e_3,e_7\rangle$\\
 \hline
(16,1) &   $\langle e_1,e_7\rangle$\\
 \hline
\end{tabular}
\item Three-dimensional subspaces:\\
\begin{tabular}{ |c|c|}
 \hline
Sequence of eigenvalues&      Three-dimensional \\
$\lambda=(\lambda_1,\lambda_2)$ for $\{\bigwedge^3 \bar A_i \}_{ i=1}^2$ &  invariant  subspace\\
 \hline
(8,1) &  $\langle e_4, e_5,e_6\rangle$\\
 \hline
(12,1) &   $\langle e_3,e_4,e_6\rangle,\langle e_3, e_5+\alpha e_4,e_6\rangle$\\
 \hline
(16,1) &  $\langle e_4, e_6,e_7\rangle,\langle e_5+\alpha e_4, e_6,e_7\rangle$\\
 \hline
(18,1) &  $\langle e_2, e_3, e_5\rangle, \langle e_2, e_3, e_6+\alpha e_5\rangle$\\
 \hline
(24,1) &  $\langle e_3, e_5 ,e_7\rangle, \langle e_3, e_6+\alpha e_5,e_7\rangle$\\
 \hline
(32,1)  &  $\langle e_1,e_5,e_7\rangle, \langle e_1, e_6+\alpha e_5,e_7\rangle$\\
 \hline
(36,1)  &  $\langle e_2, e_3,e_7\rangle$\\
 \hline
(48,1)  &  $\langle e_1, e_3,e_7\rangle$\\
 \hline
\end{tabular}
\item  Four-dimensional subspaces:\\
\begin{tabular}{ |c|c|}
 \hline
Sequence of eigenvalues&     Four-dimensional \\
$\lambda=(\lambda_1,\lambda_2)$ for $\{\bigwedge^4 \bar A_i \}_{ i=1}^2$  & invariant  subspace \\
 \hline
(24,1) &  $\langle e_3, e_4,e_5,e_6\rangle$\\
 \hline
(32,1)  &  $\langle e_4, e_5,e_6,e_7\rangle$\\
 \hline
(36,1) &   $\langle e_2, e_3,e_4,e_6\rangle,\langle e_2, e_3,e_5+\alpha e_4,e_6\rangle$\\
 \hline
(48,1) &   $\langle e_3, e_4,e_6,e_7\rangle,\langle e_3,  e_5+\alpha e_4,e_6,e_7\rangle$\\
 \hline
(64,1) &  $\langle e_1, e_4,e_6,e_7\rangle,\langle e_1,e_5+\alpha e_4,e_6,e_7\rangle$\\
 \hline
(72,1) &  $\langle e_2, e_3,e_5,e_7\rangle,\langle e_2, e_3,e_6+\alpha e_5,e_7\rangle$\\
 \hline
(96,1) &   $\langle e_1, e_3, e_5,e_7\rangle,\langle e_1, e_3, e_6+\alpha e_5,e_7\rangle$\\
 \hline
(144,1) &   $\langle e_1, e_2,e_3,e_7\rangle$\\
 \hline
\end{tabular}
\item  Five-dimensional subspaces:\\
\begin{tabular}{ |c|c| c|}
 \hline
Sequence of eigenvalues &   Five-dimensional \\
$\lambda=(\lambda_1,\lambda_2)$ for $\{\bigwedge^5 \bar A_i \}_{ i=1}^2$ & invariant  subspace \\
 \hline
(72,1) &   $\langle e_2, e_3,e_4,e_5,e_6\rangle$\\
 \hline
(96,1) &   $\langle e_3, e_4,e_5,e_6,e_7\rangle$\\
 \hline
(128,1) &  $\langle e_1, e_4,e_5,e_6,e_7\rangle$\\
 \hline
(144,1) &  $\langle e_2, e_3,e_4,e_6,e_7\rangle,\langle e_2, e_3,e_5+\alpha e_4,e_6,e_7\rangle$\\
 \hline
(192,1) &   $\langle e_1, e_3, e_4,e_6,e_7\rangle,\langle e_1, e_3,e_5+\alpha e_4,e_6,e_7\rangle$\\
 \hline
(288,1) &   $\langle e_1, e_2,e_3,e_5,e_7\rangle, \langle e_1, e_2,e_3, e_6+\alpha e_5,e_7\rangle$\\
 \hline
\end{tabular}
\item  Six-dimensional subspaces:\\
\begin{tabular}{ |c|c|}
\hline
Sequence of eigenvalues &      Six-dimensional \\
$\lambda=(\lambda_1,\lambda_2)$ for $\{\bigwedge^6 \bar A_i \}_{ i=1}^2$  & invariant  subspace \\
 \hline
(288,1) &   $\langle e_2, e_3,e_4, e_5,e_6, e_7\rangle$\\
 \hline
(384,1) &   $\langle e_1,e_3, e_4,e_5,e_6, e_7\rangle$\\
 \hline
(576,1) &  $\langle e_1, e_2,e_3, e_4,e_6, e_7\rangle,\langle e_1, e_2,e_3, e_5+\alpha e_4,e_6, e_7\rangle$\\
 \hline
\end{tabular}
\item  Seven-dimensional subspaces:\\
\begin{tabular}{ |c|c|}
\hline
Sequence of eigenvalues&   Seven-dimensional \\
$\lambda=(\lambda_1,\lambda_2)$ for $\{\bigwedge^7 \bar A_i \}_{ i=1}^2$ & invariant  subspace\\
 \hline
(1152, 1) &  $\langle e_1, e_2,e_3,e_4,e_5,e_6,e_7\rangle$\\
 \hline
\end{tabular}
 \end{itemize}
\end{example}
\begin{example}
\begin{equation}
A_1= \left( \begin {array}{ccccccc}
3&0&0&0&0&0&0\\ \noalign{\medskip}
0&2&0&0&0&0&0\\ \noalign{\medskip}
0&0&2&0&0&0&0\\ \noalign{\medskip}
0&0&0&1&0&0&0\\ \noalign{\medskip}
0&0&0&0&1&0&0\\ \noalign{\medskip}
0&0&0&0&0&1&0\\ \noalign{\medskip}
0&0&0&0&0&0&3
\end {array} \right),
A_2=\left( \begin {array}{ccccccc}
0&0&0&0&0&0&0\\ \noalign{\medskip}
0&0&0&1&0&0&0\\ \noalign{\medskip}
0&1&0&0&0&0&0\\ \noalign{\medskip}
0&0&0&0&0&0&0\\ \noalign{\medskip}
0&0&0&0&1&0&0\\ \noalign{\medskip}
0&0&0&0&1&0&0\\ \noalign{\medskip}
1&0&0&0&0&0&0\end {array} \right) .
\end{equation}
Using algorithms $A$ and $B$ with $s=1$, the complete list of the common invariant subspaces for the set of matrices $\{A_1,A_2\}$  is as follows:
 \begin{itemize}
\item Zero-dimensional subspaces:\\
$\{\bf 0\}$
\item  One-dimensional subspaces:\\
\begin{tabular}{ |c|c| }
 \hline
Sequence of eigenvalues&   One-dimensional \\
$\lambda=(\lambda_1,\lambda_2)$ for $\{ \bar A_i \}_{ i=1}^2$  &invariant subspace\\
 \hline
 (2,1)& $\langle e_6\rangle$ \\
 \hline
  (2,2)& $\langle e_5+e_6\rangle$ \\
 \hline
 (3,1) & $\langle e_3\rangle$\\
 \hline
 (4,1) & $\langle e_7\rangle$ \\
 \hline
\end{tabular}
\item  Two-dimensional subspaces:\\
\begin{tabular}{ |c|c|}
 \hline
Sequence of eigenvalues&     Two-dimensional \\
$\lambda=(\lambda_1,\lambda_2)$ for $\{\bigwedge^2 \bar A_i \}_{ i=1}^2$ & invariant  subspace \\
 \hline
(4,2) &  $\langle e_5,e_6\rangle$\\
 \hline
(6,1) &  $\langle e_3,e_6\rangle$\\
 \hline
(6,2) &  $\langle e_3,e_5+e_6\rangle$\\
 \hline
(8,1) &  $\langle e_6,e_7\rangle$\\
 \hline
(8,2) &  $\langle e_5+e_6,e_7\rangle$\\
 \hline
(9,1) &  $\langle e_2,e_3\rangle$\\
 \hline
(12,1) &  $\langle e_3,e_7\rangle$\\
 \hline
(16,1) &  $\langle e_1,e_7\rangle$\\
 \hline
\end{tabular}
\item Three-dimensional subspaces:\\
\begin{tabular}{ |c|c|}
 \hline
Sequence of eigenvalues&  Three-dimensional \\
$\lambda=(\lambda_1,\lambda_2)$ for $\{\bigwedge^3 \bar A_i \}_{ i=1}^2$& invariant  subspace \\
 \hline
(12,2) & $\langle e_3, e_5,e_6\rangle$\\
 \hline
(16,2) &   $\langle e_5, e_6,e_7\rangle$\\
 \hline
(18,1) &  $\langle e_2, e_3,e_4\rangle, \langle e_2, e_3,e_6+\alpha e_4\rangle$\\
 \hline
(18,2) &  $\langle e_2, e_3,e_5+e_6\rangle$\\
 \hline
(24,1) &   $\langle e_3, e_6,e_7\rangle$\\
 \hline
(24,2) &   $\langle e_3, e_5+e_6,e_7\rangle$\\
 \hline
(32,1) &  $\langle e_1, e_6,e_7\rangle$\\
 \hline
(32,2) &  $ \langle e_1, e_5+e_6,e_7\rangle$\\
 \hline
(36,1) &   $\langle e_2, e_3,e_7\rangle$\\
 \hline
(48,1) &   $\langle e_1, e_3,e_7\rangle$\\
 \hline
\end{tabular}
\item  Four-dimensional subspaces:\\
\begin{tabular}{ |c|c|}
 \hline
Sequence of eigenvalues &      Four-dimensional \\
$\lambda=(\lambda_1,\lambda_2)$ for $\{\bigwedge^4 \bar A_i \}_{ i=1}^2$ & invariant  subspace \\
 \hline
(36,1) &  $\langle e_2, e_3,e_4,e_6\rangle,$\\
 \hline
(36,2)     &  $\langle e_2, e_3,e_4,e_5+e_6\rangle,\langle e_2, e_3, e_5+\alpha e_4,e_6-\alpha e_4\rangle$\\
 \hline
(48,2) &  $\langle e_3, e_5,e_6,e_7\rangle$\\
 \hline
(64,2) &   $\langle e_1, e_5,e_6,e_7\rangle$\\
 \hline
  (72,1)         & $\langle e_2, e_3,e_4,e_7\rangle, \langle e_2, e_3, e_6+\alpha e_4,e_7\rangle$\\
 \hline
(72,2) & $\langle e_2, e_3,e_5+e_6,e_7\rangle$\\
 \hline
(96,1) &  $\langle e_1, e_3,e_6,e_7\rangle$\\
 \hline
(96,2) &  $\langle e_1, e_3,e_5+e_6,e_7\rangle$\\
 \hline
(144,1) &  $\langle e_1, e_2,e_3,e_7\rangle$\\
 \hline
\end{tabular}
\item  Five-dimensional subspaces:\\
\begin{tabular}{ |c|c|}
 \hline
Sequence of eigenvalues &       Five-dimensional \\
$\lambda=(\lambda_1,\lambda_2)$ for $\{\bigwedge^5 \bar A_i \}_{ i=1}^2$ &invariant  subspace\\
 \hline
(72,2) &  $\langle e_2, e_3,e_4,e_5,e_6\rangle$\\
 \hline
(144,1) &  $\langle e_2, e_3,e_4,e_6,e_7\rangle,$\\
 \hline
  (144,2)        & $\langle e_2, e_3, e_4,e_5+e_6,e_7\rangle,\langle e_2, e_3,e_5+\alpha e_4,e_6-\alpha e_4,e_7\rangle$\\
 \hline
(192,2) &   $\langle e_1, e_3,e_5,e_6,e_7\rangle$\\
 \hline
 (288,1)       &   $\langle e_1, e_2,e_3, e_4,e_7\rangle,\langle e_1, e_2,e_3,e_6+\alpha e_4,e_7\rangle$\\
 \hline
(288,2) &  $\langle e_1, e_2,e_3,e_5+e_6,e_7\rangle$\\
 \hline
\end{tabular}
\item  Six-dimensional subspaces:\\
\begin{tabular}{ |c|c|}
\hline
Sequence of eigenvalues&      Six-dimensional \\
$\lambda=(\lambda_1,\lambda_2)$ for $\{\bigwedge^6 \bar A_i \}_{ i=1}^2$ &  invariant  subspace \\
 \hline
(288,2) &   $\langle e_2,e_3, e_4,e_5,e_6, e_7\rangle$\\
 \hline
(576,1) &   $\langle e_1, e_2,e_3, e_4,e_6, e_7\rangle$\\
 \hline
 (576,2)     &   $\langle e_1, e_2,e_3, e_4,e_5+e_6, e_7\rangle,$\\
& $\langle e_1, e_2,e_3, ,e_5+\alpha e_4,e_6-\alpha e_4,e_7\rangle$\\
 \hline
\end{tabular}
\item  Seven-dimensional subspaces:\\
\begin{tabular}{ |c|c|}
\hline
Sequence of eigenvalues&    Seven-dimensional \\
$\lambda=(\lambda_1,\lambda_2)$ for $\{\bigwedge^7 \bar A_i \}_{ i=1}^2$   &invariant  subspace\\
 \hline
(1152, 2) &   $\langle e_1, e_2,e_3,e_4,e_5,e_6,e_7\rangle$\\
 \hline
\end{tabular}
 \end{itemize}
\end{example}
\begin{example}
\begin{equation}
\begin {array}{ll}
A_1= \left( \begin {array}{ccccccccc}
0&0&0&0&0&0&0&0&0\\ \noalign{\medskip}
0&2&0&0&0&0&0&0&0\\ \noalign{\medskip}
0&0&1&0&0&0&0&0&0\\ \noalign{\medskip}
0&0&0&-2&0&0&0&0&0\\ \noalign{\medskip}
0&0&0&0&0&0&0&0&0\\ \noalign{\medskip}
0&0&0&0&0&-1&0&0&0\\ \noalign{\medskip}
0&0&0&0&0&0&-1&0&0\\ \noalign{\medskip}
0&0&0&0&0&0&0&1&0\\ \noalign{\medskip}0&0&0&0&0&0&0&0&0
\end {array} \right),
A_2=\left( \begin {array}{ccccccccc}
 0&0&0&1&0&0&0&0&0\\ \noalign{\medskip}
-1&0&0&0&1&0&0&0&0\\ \noalign{\medskip}
0&0&0&0&0&1&0&0&0\\ \noalign{\medskip}
0&0&0&0&0&0&0&0&0\\ \noalign{\medskip}
0&0&0&-1&0&0&0&0&0\\ \noalign{\medskip}
0&0&0&0&0&0&0&0&0\\ \noalign{\medskip}
0&0&0&0&0&0&0&0&0\\ \noalign{\medskip}
0&0&0&0&0&0&-1&0&0\\ \noalign{\medskip}
0&0&0&0&0&0&0&0&0
\end {array} \right),\\
A_3=\left( \begin {array}{ccccccccc}
 0&-1&0&0&0&0&0&0&0\\ \noalign{\medskip}
0&0&0&0&0&0&0&0&0\\ \noalign{\medskip}
0&0&0&0&0&0&0&0&0\\ \noalign{\medskip}
1&0&0&0&-1&0&0&0&0\\ \noalign{\medskip}
0&1&0&0&0&0&0&0&0\\ \noalign{\medskip}
0&0&1&0&0&0&0&0&0\\ \noalign{\medskip}
0&0&0&0&0&0&0&-1&0\\ \noalign{\medskip}
0&0&0&0&0&0&0&0&0\\ \noalign{\medskip}
0&0&0&0&0&0&0&0&0\end {array} \right).
\end{array}
\end{equation}
Using algorithms $A$ and $B$ with $s=3$, the complete list of the common invariant subspaces for the set of matrices $\{A_1,A_2,A_3\}$  is as follows:
 \begin{itemize}
\item Zero-dimensional subspaces:\\
$\{\bf 0\}$
\item  One-dimensional subspaces:\\
\begin{tabular}{ |c|c| }
 \hline
Sequence of eigenvalues&   One-dimensional \\
$\lambda=(\lambda_1,\lambda_2,\lambda_3)$ for $\{ \bar A_i \}_{ i=1}^3$  &invariant subspace\\
 \hline
 (3, 3, 3)& $\langle e_1+e_5\rangle, \langle e_9+\alpha(e_1+e_5) \rangle$ \\
 \hline
\end{tabular}
\item  Two-dimensional subspaces:\\
\begin{tabular}{ |c|c|}
 \hline
Sequence of eigenvalues&     Two-dimensional \\
$\lambda=(\lambda_1,\lambda_2,\lambda_3)$ for $\{\bigwedge^2 \bar A_i \}_{ i=1}^3$ & invariant  subspace \\
 \hline
(8, 9, 9)  &  $\langle e_3, e_6\rangle, \langle  e_7+\alpha e_6, e_8-\alpha e_3\rangle$\\
 \hline
(9, 9, 9)  &  $\langle e_1+e_5,e_9\rangle$\\
 \hline
\end{tabular}
\item Three-dimensional subspaces:\\
\begin{tabular}{ |c|c|}
 \hline
Sequence of eigenvalues&       Three-dimensional \\
$\lambda=(\lambda_1,\lambda_2,\lambda_3)$ for $\{\bigwedge^3 \bar A_i \}_{ i=1}^3$& invariant  subspace \\
 \hline
(24, 27, 27)    & $\langle e_3,e_6,e_9+\alpha (e_1+e_5)\rangle,$\\
 & $\langle e_3,e_1+e_5,e_6\rangle,$\\
 & $\langle e_7, e_8, e_9+\alpha (e_1+e_5)\rangle,$\\
  & $ \langle e_1+e_5, e_7+\alpha e_6, e_8-\alpha e_3\rangle,$\\
  & $ \langle e_7+\alpha e_6, e_8-\alpha e_3,e_9\rangle,$\\
  & $ \langle e_7+\alpha e_6, e_8-\alpha e_3,e_9+\alpha (e_1+e_5)\rangle$\\
 \hline
 (15, 27, 27)& $ \langle e_2,e_4,e_5- e_1\rangle$\\
 \hline
\end{tabular}
\item  Four-dimensional subspaces:\\
\begin{tabular}{ |c|c|}
 \hline
Sequence of eigenvalues &      Four-dimensional \\
$\lambda=(\lambda_1,\lambda_2,\lambda_3)$ for $\{\bigwedge^4 \bar A_i \}_{ i=1}^3$ & invariant  subspace \\
 \hline
(64, 81, 81) &   $\langle e_3,e_6,e_7,e_8\rangle$\\
 \hline
(72, 81, 81) &   $\langle e_3, e_1+e_5, e_6,e_9\rangle,$\\
&   $\langle e_1+e_5, e_7+\alpha e_6, e_8-\alpha e_3,e_9\rangle$\\
 \hline
(45, 81, 81)& $\langle e_1,e_2,e_4,e_5\rangle,$\\
 & $ \langle e_2,e_4,e_5- e_1,e_9+\alpha e_1\rangle$\\
 \hline
\end{tabular}
\item  Five-dimensional subspaces:\\
\begin{tabular}{ |c|c|}
 \hline
Sequence of eigenvalues &       Five-dimensional \\
$\lambda=(\lambda_1,\lambda_2,\lambda_3)$ for $\{\bigwedge^5 \bar A_i \}_{ i=1}^3$ &invariant  subspace\\
 \hline
(192, 243, 243)  &  $\langle e_3,e_1+e_5,e_6,e_7,e_8\rangle,$\\
& $\langle e_3,e_6,e_7,e_8,e_9+\alpha ( e_1+e_5)\rangle$\\
 \hline
(120, 243, 243)  &   $\langle e_2,e_3,e_4,e_5- e_1,e_6\rangle, $\\
                         & $\langle e_2,e_4,e_5- e_1,e_7+\alpha e_6 ,e_8-\alpha e_3\rangle$\\
 \hline
  (135, 243, 243)        &   $\langle e_1,e_2,e_4,e_5,e_9\rangle$\\
 \hline
\end{tabular}
\item  Six-dimensional subspaces:\\
\begin{tabular}{ |c|c|}
\hline
Sequence of eigenvalues&      Six-dimensional \\
$\lambda=(\lambda_1,\lambda_2,\lambda_3)$ for $\{\bigwedge^6 \bar A_i \}_{ i=1}^3$ &  invariant  subspace \\
 \hline
(576, 729, 729)  &  $\langle e_3,e_1+e_5, e_6,e_7,e_8,e_9\rangle$\\
 \hline
 (360, 729, 729)          &   $\langle e_1,e_2,e_3,e_4, e_5,e_6\rangle,$\\
&   $\langle e_1, e_2, e_4, e_5, e_7+\alpha e_6, e_8-\alpha e_3\rangle,$\\
&   $\langle e_2,e_3,e_4,e_5- e_1,e_6,e_9+\alpha e_1\rangle,$\\
& $\langle e_2, e_4, e_5+e_1, e_7+\alpha e_6, e_8-\alpha e_3, e_9+\beta e_1\rangle$\\
 \hline
\end{tabular}
\item  Seven-dimensional subspaces:\\
\begin{tabular}{ |c|c|}
\hline
Sequence of eigenvalues&    Seven-dimensional \\
$\lambda=(\lambda_1,\lambda_2,\lambda_3)$ for $\{\bigwedge^7 \bar A_i \}_{ i=1}^3$   &invariant  subspace\\
 \hline
 (960, 2187, 2187)&   $\langle e_2,e_3,e_4,e_5- e_1,e_6,e_7,e_8\rangle$\\
 \hline
 (1080, 2187, 2187)    &   $\langle e_1,e_2,e_3,e_4, e_5,e_6,e_9\rangle,$\\
    &   $\langle e_1, e_2, e_4, e_5, e_7+\alpha e_6, e_8 -\alpha e_3, e_9\rangle$\\
 \hline
\end{tabular}
\item  Eight-dimensional subspaces:\\
\begin{tabular}{ |c|c|}
\hline
Sequence of eigenvalues&    Eight-dimensional \\
$\lambda=(\lambda_1,\lambda_2,\lambda_3)$ for $\{\bigwedge^8 \bar A_i \}_{ i=1}^3$   &invariant  subspace\\
 \hline
 (2880, 6561, 6561)&   $\langle e_1,e_2,e_3,e_4, e_5,e_6,e_7,e_8\rangle,$\\
           &   $\langle e_2,e_3,e_4, e_5-e_1,e_6,e_7,e_8,e_9+\alpha e_1\rangle$\\
 \hline
\end{tabular}
\item  Nine-dimensional subspaces:\\
\begin{tabular}{ |c|c|}
\hline
Sequence of eigenvalues&    Nine-dimensional \\
$\lambda=(\lambda_1,\lambda_2,\lambda_3)$ for $\{\bigwedge^9 \bar A_i \}_{ i=1}^3$   &invariant  subspace\\
 \hline
(8640, 19683, 19683) &   $\langle e_1,e_2,e_3,e_4, e_5,e_6,e_7,e_8,e_9\rangle$\\
 \hline
\end{tabular}
 \end{itemize}
\end{example}

\subsection*{Acknowledgments}
Ahmad Y. Al-Dweik would like to thank Qatar University for its support and excellent research facilities. R. Ghanam and G. Thompson are grateful to VCU Qatar and Qatar Foundation for their support.


\begin{thebibliography}{99}

\bibitem[BF]{BF} L. Brickman and P. A. Fillmore, {\it The Invariant Subspace Lattice of a Linear Transformation}, Canadian Journal of Mathematics {\bf 19}, 810-822, 1967.
\bibitem[CGLMP]{CGLMP} C. Chen, C, O. Golubitsky, F. Lemaire, M. Moreno Maza and W. Pan,, {\it Comprehensive Triangular Decomposition}, Proc. CASC 2007, LNCS Vol. 4770, pp. 73-101. Springer, 2007.
\bibitem[E]{E} P. Enflo,  {\it On the invariant subspace problem for Banach spaces}, Acta Mathematica, {\bf 158}(3), 213-313, (1987).
\bibitem[GLR]{GLR} I. Gohberg, P. Lancaster, L. Rodman,  {\it Invariant Subspaces of Matrices with Applications}, Wiley-Interscience, New York, 1986.
\bibitem[GVL]{GVL} G.H. Golub, C.F. Van Loan,  {\it Matrix Computations}, John Hopkins University Press, Baltimore, MD, 1989.
\bibitem[H]{H} J. Harris, {\it Algebraic Geometry, A first course}, Graduate Texts in Mathematics, Springer, New York 1992.
\bibitem[HP]{HP} W.V.D. Hodge, P.D. Pedoe, {\it  Methods of Algebraic Geometry}, Cambridge University Press, Cambridge, MA, 1952.
\bibitem[M]{M} D. Mumford, {\it Algebraic Geometry I: Complex Projective Varieties}, Grundlehren der math. Wissenschaften  {\bf 221}, Springer New York, 1976.
\bibitem[T]{T} M. Tsatsomeros,  {\it A criterion for the existence of common invariant subspaces of matrices}, Linear Algebra and its Applications  {\bf 322}(1-3), 51-59, 2001.
\bibitem[Y]{Y}  B. S. Yadav, {\it The present state and heritages of the invariant subspace problem}, Milan Journal of Mathematics, {\bf 73}(1), 289-316, 2005. 







\end{thebibliography}
\end{document}